\documentclass[12pt]{amsart}

\usepackage[utf8]{inputenc}      % Codificación UTF-8
\usepackage[T1]{fontenc}
\usepackage{lmodern}
\usepackage[english]{babel}
\usepackage{csquotes}

\usepackage{setspace}
\usepackage{amsmath, amssymb, amsthm, mathtools}
\usepackage{mathrsfs}            % Letras caligráficas
\usepackage{esint}               % Integrales dobles, etc.

\usepackage{enumitem}            % Mejoras en listas enumeradas
\usepackage{graphicx}            % Imágenes
\usepackage{xcolor}              % Colores
\usepackage{hyperref}            % Hipervínculos
\hypersetup{
  colorlinks = true,
  linkcolor = blue,
  citecolor = blue,
  urlcolor  = blue
}

\usepackage[margin=1.2in]{geometry}

\newtheorem{theorem}{Theorem}[section]
\newtheorem{lemma}[theorem]{Lemma}

\theoremstyle{definition}
\newtheorem{definition}[theorem]{Definition}

\theoremstyle{remark}
\newtheorem{remark}[theorem]{Remark}

\def\dis{\displaystyle}
\def\bq{\begin{equation}}
\def\eq{\end{equation}}
\def\bp{\begin{proof}}
\def\ep{\end{proof}}
\def\noi{\noindent}

\def\S{{ S^{k}_{ij}}}

\def\Cn{{\mathbb C}^n}

\numberwithin{equation}{section}

\title[Schwarzian Derivative and Convexity]{The Schwarzian Derivative for Convex Holomorphic Mappings in Several Complex Variables}
\author{Rodrigo Hern\'andez}
\date{}

\begin{document}

\begin{abstract}
We study the two components of Oda's Schwarzian derivative, the tensor
$S_f$ and the quadratic form $S^0_f$, for convex holomorphic mappings of the
polydisk and for Roper--Suffridge extensions of convex functions to the unit
ball of $\mathbb{C}^n$, with norms taken in the Bergman metric. For the polydisk
we obtain the sharp bound $\|S_f\|\le\gamma_n$, where $\gamma_n$ is the maximum
of an explicit quartic form on the unit sphere of $\mathbb{C}^n$, with
$\gamma_2=2$ and $\gamma_n\to2\sqrt2$, and the sharp bound $\|S^0_f\|\le 4/(n+1)$, $n\ge2$, in both cases the
extremal mappings are products of half-plane mappings. For Roper--Suffridge extensions
both norms are computed exactly in terms of the Schwarzian derivative of the
one-variable function. In particular $S^0_f$ reduces to the classical
Schwarzian derivative, and the theorem of Chuaqui, Duren and Osgood yields the
sharp bounds $\|S_f\|\le1/\sqrt{n+1}$ and $\|S^0_f\|\le 1/(n+1)$, with equality
for the parallel strip.
\end{abstract}

\subjclass[2020]{Primary 32H02; Secondary 30C45, 30C55, 30C80}
 
\keywords{Schwarzian derivative, convex mappings, polydisk, unit ball,
Roper--Suffridge operator, Bergman metric, several complex variables}
\maketitle

\section*{Introduction}

For a locally univalent holomorphic function $\varphi$ in a domain of the
complex plane, the Schwarzian derivative is
\[
S\varphi = \left( \frac{\varphi''}{\varphi'} \right)' - \frac{1}{2}
\left( \frac{\varphi''}{\varphi'} \right)^2 ,
\]
and $S\varphi\equiv0$ exactly when $\varphi$ is a M\"{o}bius transformation.
Nehari \cite{nehari1949} proved that the condition
\[
|S\varphi(z)| \leq \frac{2}{(1 - |z|^2)^2},\qquad z\in\mathbb{D},
\]
implies the univalence of $\varphi$ in the unit disk $\mathbb{D}$, and that the
constant $2$ is sharp.

For convex mappings the Schwarzian derivative is bounded. Chuaqui, Duren and
Osgood \cite{chuaqui2011} proved that if $\varphi$ maps $\mathbb{D}$
conformally onto a convex domain, then
\[
\|S\varphi\| = \sup_{z \in \mathbb{D}} (1 - |z|^2)^2 |S\varphi(z)| \leq 2,
\]
with equality precisely when $\varphi$ maps $\mathbb{D}$ onto a parallel
strip; in particular $\|S\varphi\|<2$ for bounded convex mappings. Their
argument rests on the representation $\varphi''/\varphi'=2\omega/(1-z\omega)$,
where $\omega$ is holomorphic in $\mathbb{D}$ with $|\omega|\le1$, which gives
\[
S\varphi(z) = \frac{2\omega'(z)}{(1 - z\omega(z))^2}.
\]
Sharp norm estimates of this kind have also been obtained for other
subclasses of convex functions; see, for instance, \cite{kanas2011}.

In several complex variables the Schwarzian derivative is no longer a single
scalar function, and several non-equivalent generalizations have been proposed
\cite{MM1, MM2, MT}. In \cite{He, hernandez2007} we studied the one introduced
by Oda \cite{Oda}, a tensor of type $(1,2)$ that is invariant under
post-composition with M\"{o}bius transformations, and showed that bounds on
its norm imply univalence, in analogy with Nehari's criterion. In a related
direction, joint work with Mart\'{\i}n \cite{hernandezmartin2013} introduced
pre-Schwarzian and Schwarzian derivatives for complex harmonic mappings.

The present paper takes up the convexity problem in this framework. Oda's
Schwarzian derivative of a locally biholomorphic mapping $f$ consists of two
pieces: the coefficients $S^k_{ij}f$, which we collect into the tensor $S_f$,
and the coefficients $S^0_{ij}f$, which form a quadratic form $S^0_f$. They
play different roles. When $n=1$ the tensor vanishes identically, and the
classical Schwarzian derivative is carried by the quadratic form: for a
function $\varphi$ of one variable, $S^0_{11}\varphi=-\frac12 S\varphi$. For
$n\ge2$, the tensor $S_f$ controls univalence \cite{He,hernandez2007}, while
$S^0_f$ is determined by $S_f$ and its first derivatives; see the proof of
Lemma~3.6 in \cite{He}. It is therefore natural to ask for sharp bounds of both
norms for convex mappings.

We work in two settings. For convex mappings of the polydisk $\mathbb{P}^n$,
which by Suffridge's theorem are products of convex functions of one variable
up to a linear change of coordinates, we prove that $\|S_f\|\le\gamma_n$,
where $\gamma_n$ is the maximum of an explicit quartic form on the unit sphere
of $\mathbb{C}^n$, and that this constant is sharp. We have $\gamma_2=2$ and
\[
\max\Bigl\{2,\ 2\sqrt2\,\frac{n-1}{n+1}\Bigr\}\le\gamma_n\le
\frac{\sqrt{8(n+3)(n-1)}}{n+1},
\]
so that $\gamma_n\to2\sqrt2$ as $n\to\infty$. For the quadratic form we prove
the sharp estimate $\|S^0_f\|\le 4/(n+1)$, $n\ge2$. In both cases the extremal
configuration is produced by products of half-plane mappings, whose classical
Schwarzian derivatives vanish, a phenomenon without counterpart in one
variable. The bound for $S^0_f$ comes from the pre-Schwarzian terms in
$S^0_f$, which cancel only when $n=1$, and the proof rests on the inequality
$(1-|z|^2)^2\bigl(|\varphi''/\varphi'|^2+2|S\varphi|\bigr)<16$ for convex
functions $\varphi$, which follows from Trimble's coefficient inequality
\cite{Tri75}.

For the unit ball $\mathbb{B}^n$ convex mappings are considerably harder to
parametrize, and we restrict ourselves to the Roper--Suffridge extensions
$f=\Phi(\varphi)$ of convex functions $\varphi$ of the disk. Within this class
both norms can be computed exactly:
\[
\|S_f\|=\frac{\|S\varphi\|}{2\sqrt{n+1}},\qquad
\|S^0_f\|=\frac{\|S\varphi\|}{2(n+1)} .
\]
The second identity reflects the fact that $S^0_f$ is the classical Schwarzian
derivative of $\varphi$ in the $z_1$ direction. Combined with the theorem of
Chuaqui, Duren and Osgood, these identities give the sharp bounds
$\|S_f\|\le1/\sqrt{n+1}$ and $\|S^0_f\|\le1/(n+1)$, with equality for the
parallel strip. For $n=1$ the latter reduces to the one-variable theorem,
written in the Bergman metric of the disk.

Section~\ref{sec:prelim} contains the definitions. Section~\ref{sec:main} is
devoted to $S_f$, and Section~\ref{sec:S0} to $S^0_f$.

\section{Preliminaries}\label{sec:prelim}

Let $f:\Omega\subset\Cn\rightarrow\Cn$ be a locally
biholomorphic mapping defined on some domain $\Omega$. Several non-equivalent
definitions of the Schwarzian derivative in several complex variables have
been proposed in the literature \cite{MM1, MM2, MT}; we work with the one
introduced in \cite{Oda}, which has the advantage of being tied directly to
the theory of completely integrable systems. Following \cite{Oda}, the
Schwarzian derivatives of $f=(f_1,\ldots,f_n)$ are defined as
\[
\S f= \dis\sum^{n}_{l=1}\frac{\partial^{2} f_{l}}{\partial z_{i}\partial
z_{j}} \frac{\partial z_{k}}{\partial
f_{l}}-\frac{1}{n+1}\left(\delta^{k}_{i}\frac{\partial}{\partial
z_{j}}+\delta^{k}_{j}\frac{\partial}{\partial z_{i}}\right)
\log(J_f)\, ,
\]
where $i,j,k=1,2, \ldots,n$ and $\delta^{k}_{i}$ are the Kronecker symbols.
For $n>1$, the Schwarzian derivatives satisfy $\S f=0$ for all
$i,j,k=1,2,\ldots,n$ if and only if $f(z)=M(z)$ for some M\"{o}bius
transformation \cite{Oda,He}
\[
M(z)=\left(\frac{l_1(z)}{l_0(z)},\ldots,\frac{l_n(z)}{l_0(z)}\right),
\]
where $l_i(z)=a_{i0}+a_{i1}z_1+\cdots+a_{in}z_n$ with $\det(a_{ij})\neq 0$.
Furthermore, for a composition one has
\bq\label{chain rule}
\S(g\circ f)(z)=\S f(z)+
\sum^{n}_{l,m,r=1}S^{r}_{lm}g(w)\frac{\partial w_{l}}{\partial z_{i}}
\frac{\partial w_{m}}{\partial z_{j}}
\frac{\partial z_{k}}{\partial w_{r}}\; ,\quad w=f(z)\, .
\eq
In particular, if $g$ is a M\"{o}bius transformation, then
$\S(g\circ f)=\S f$. The coefficients $S^0_{ij}f$ are given by
\[
S^0_{ij}f(z)=(J_f)^{\frac{1}{n+1}}\left(\frac{\partial^2}{\partial
z_i\partial z_j}(J_f)^{-\frac{1}{n+1}}-
\sum_{k=1}^n\frac{\partial}{\partial
z_k}(J_f)^{-\frac{1}{n+1}}S^k_{ij}f(z)\right).
\]
In \cite{Oda} one also finds a description of the functions with prescribed
Schwarzian derivatives $\S f$. Consider the following over-determined system
of partial differential equations:
\bq\label{system}
\frac{\partial^{2}u}{\partial z_{i}\partial z_{j}} =
\sum^{n}_{k=1}P^{k}_{ij}(z)\frac{\partial u}{\partial z_{k}}+
P^{0}_{ij}(z)u \; ,\quad i,j= 1,2, \ldots,n\, ,
\eq
where $z=(z_{1},z_{2},\ldots,z_n)\in \Omega$ and $P^k_{ij}(z)$ are holomorphic
functions in $\Omega$ for $1\le i,j\le n$ and $0\le k\le n$. The system \eqref{system} is
called \textit{completely integrable} if it admits $n+1$ (the maximum number
of) linearly independent solutions, and is said to be in \textit{canonical
form} (see \cite{yoshida1976}) if the coefficients satisfy
\[
\sum_{j=1}^{n} P_{ij}^{j}(z)=0\, ,\quad i=1,2,\ldots,n\, .
\]
It was shown in \cite{Oda} that \eqref{system} is a completely integrable system in canonical
form if and only if $P^k_{ij}=\S f$ for a locally biholomorphic mapping
$f=(f_1, \ldots, f_n)$, where $f_i=u_i/u_0$ for $1\leq i\leq n$ and
$u_0, u_1,\ldots, u_n$ is a set of linearly independent solutions of the
system. For a given mapping $f$, the function $u_0=(J_f)^{-\frac{1}{n+1}}$ is
always a solution of \eqref{system} with $P_{ij}^k=\S f$.

We recall the following definitions from \cite{He}, in which the individual
Schwarzian derivatives $S^kf$ are combined to define the operators below.

\begin{definition}
For each $k=1,\ldots,n$, we let $S^{k}_f$ denote the $n\times n$ matrix
\[
S^k_f= (\S f)\, ,\quad i,j=1,\ldots,n\, .
\]
\end{definition}

\begin{definition}
Let $f:\Omega\rightarrow\Cn$ be locally biholomorphic, and let $T_z\Omega$
denote the tangent space at $z\in\Omega$. We define the \textit{Schwarzian
derivative operator} as the quadratic mapping
$\vec v\mapsto S_f(z)(\vec{v},\vec{v})$ from $T_z\Omega$ to $T_z\Omega$ given by
\[
S_f(z)(\vec{v},\vec{v})=\left(S^1_f(z)(\vec{v}),\ldots,S^{n}_f(z)(\vec{v})\right),
\]
where $\vec{v}\in T_z \Omega$ and $S^k_f(z)(\vec{v})=\vec{v}^{\,t}S^k_f(z)\vec{v}$.
\end{definition}

\noi With this notation, equation \eqref{system} with $P^k_{ij}=\S f$ can be
rewritten as
\bq\label{Hess}
\mathrm{Hess}(u)(z)(\vec{v},\vec{v})
= S_f(z)(\vec{v},\vec{v})\cdot \nabla u(z)+S^0_f(z)(\vec{v})\,u\, ,
\eq
where $S^0_f(z)(\vec{v})=\vec{v}^{\,t}S^0_f(z)\vec{v}$.

In \cite{He} it is shown that
\begin{equation}\label{sf operator}
S_f(z)(\vec{v},\vec{v})=(Df(z))^{-1}D^2f(z)(\vec{v},\vec{v})
-\frac{2}{n+1}\bigl(\nabla\log J_f(z)\cdot\vec{v}\bigr)\,\vec{v}.
\end{equation} 

The Bergman metric in the polydisk
$\mathbb{P}^n=\{(z_1,\ldots,z_n): |z_i|<1,\,i=1,\ldots, n\}$
is the Hermitian metric defined by the diagonal matrix
\[
g_{ii}(z)=\frac{2}{(1-|z_i|^2)^2}\,.
\]
Let $\mathbb{B}^n=\{(z_1,\ldots,z_n)\in\mathbb{C}^n:|z_1|^2+\cdots+|z_n|^2<1\}$.
The Bergman metric on $\mathbb{B}^n$ is the Hermitian inner product defined by
\begin{equation}\label{bergman metric}
g_{ij}(z)=\frac{n+1}{(1-|z|^2)^2}\left[(1-|z|^2)\delta_{ij}+\bar{z}_iz_j\right],
\end{equation}
where $\delta_{ij}$ denotes the Kronecker symbol; see, e.g., \cite{krantz1982}.
The norm of the Schwarzian derivative operator of $f$ on the polydisk or the
Euclidean unit ball is given by
\[
\|S_f(z)\|=\sup_{\|\vec{v}\|=1}\|S_f(z)(\vec{v},\vec{v})\|\,,
\]
where the norm is the Bergman norm of the respective domain. Finally, we set
\[
\|S_f\| = \sup_{z\in\Omega}\|S_f(z)\|\,,\qquad \Omega=\mathbb{P}^n
\text{ or }\Omega=\mathbb{B}^n.
\]

In this paper, we use these norms to measure the size of the Schwarzian
operator for convex mappings on the polydisk and for the Roper--Suffridge
extension in the Euclidean unit ball of $\mathbb{C}^n$.

\section{Schwarzian Derivative for Convex Mappings}\label{sec:main}

We begin with the polydisk. Let $f:\mathbb{P}^n\to\mathbb{C}^n$ be a convex
biholomorphic mapping. By a theorem of Suffridge \cite{Suf70}, $f$ can be
written as
\[
f(z)=f(z_1,\ldots,z_n) = M\bigl(\varphi_1(z_1),\dots,\varphi_n(z_n)\bigr),
\]
where each $\varphi_i$ is a convex holomorphic mapping of the unit disk
$\mathbb{D}$, and $M\in\mathcal{L}(\mathbb{C}^n,\mathbb{C}^n)$ is a
nonsingular linear map; if $f(0)=0$ and $Df(0)=I$, then $M=I$ and
$\varphi_i(0)=0$, $\varphi_i'(0)=1$. Since $f=M\circ \Psi$, where $M$ is
linear and $\Psi(z)=(\varphi_1(z_1),\ldots, \varphi_n(z_n))$, the chain rule
\eqref{chain rule} gives $S_f(z)=S_\Psi(z)$. Without loss of generality,
we may therefore assume that $M$ is the identity. Straightforward calculations
then show that
\[
Df(z)=
\begin{bmatrix}
  \varphi_1'(z_1) & 0 & \cdots & 0 \\
  0 & \ddots & & \vdots \\
  \vdots & & \ddots & 0 \\
  0 & \cdots & 0 & \varphi_n'(z_n)
\end{bmatrix}.
\]
The Jacobian of $f$ is given by
$J_f(z)=\varphi_1'(z_1)\cdots\varphi_n'(z_n)$, which implies that
\[
\nabla\log J_f(z)
=\Bigl(\dfrac{\varphi_1''(z_1)}{\varphi_1'(z_1)},\dots,
\dfrac{\varphi_n''(z_n)}{\varphi_n'(z_n)}\Bigr).
\]
Thus, the Schwarzian derivative tensor applied to the vector
$\vec{v}=(v_1,\ldots,v_n)$ is given by
\[
S_f(z)(\vec{v},\vec{v})
= (Df(z))^{-1}D^2f(z)(\vec{v},\vec{v})
-\frac{2}{n+1}\bigl(\nabla\log J_f(z)\cdot\vec{v}\bigr)\,\vec{v},
\]
or equivalently,
\[
S_f(z)(\vec{v},\vec{v})
=
\begin{pmatrix}
  \displaystyle\frac{\varphi_1''}{\varphi_1'}(z_1)v_1^2
  -\displaystyle\frac{2}{n+1}
  \Bigl(\sum_{j=1}^n\frac{\varphi_j''}{\varphi_j'}(z_j)v_j\Bigr)v_1 \\
  \vdots \\
  \displaystyle\frac{\varphi_n''}{\varphi_n'}(z_n)v_n^2
  -\displaystyle\frac{2}{n+1}
  \Bigl(\sum_{j=1}^n\frac{\varphi_j''}{\varphi_j'}(z_j)v_j\Bigr)v_n
\end{pmatrix}
.
\]
For $x\in\mathbb{C}^n$ we set
\[
Q_n(x)=\sum_{k=1}^n|x_k|^2\Bigl|x_k-\frac{2}{n+1}\sum_{j=1}^nx_j\Bigr|^2,
\qquad
\gamma_n=2\sqrt2\,\max_{|x|=1}Q_n(x)^{1/2}.
\]

\begin{theorem}\label{thm:polydisk}
Let $n\geq 2$ and let $f:\mathbb{P}^n\to\mathbb{C}^n$ be a convex
biholomorphic mapping. Then
\[
\|S_f\|\leq\gamma_n ,
\]
and the constant $\gamma_n$ is sharp. Moreover, $\gamma_2=2$ and
\begin{equation}\label{eq:gammabounds}
\max\Bigl\{2,\ 2\sqrt2\,\frac{n-1}{n+1}\Bigr\}\leq\gamma_n\leq
\frac{\sqrt{8(n+3)(n-1)}}{n+1}<2\sqrt2 .
\end{equation}
\end{theorem}

\begin{proof}
Fix $z\in\mathbb{P}^n$ and write
\[
x_k=\frac{\sqrt{2}\,v_k}{1-|z_k|^2},\qquad
p_k=(1-|z_k|^2)\frac{\varphi_k''}{\varphi_k'}(z_k),
\]
so that $\vec v$ is a unit vector in the Bergman metric of the polydisk if and
only if $\sum_k|x_k|^2=1$. Since $\frac{\varphi_j''}{\varphi_j'}(z_j)v_j
=p_jx_j/\sqrt2$, the formula for $S_f(z)(\vec v,\vec v)$ above gives
\begin{equation}\label{eq:normx}
\|S_f(z)(\vec v,\vec v)\|^2=\frac12\sum_{k=1}^n|x_k|^2
\Bigl|p_kx_k-\frac{2}{n+1}\sum_{j=1}^np_jx_j\Bigr|^2 .
\end{equation}
Since each $\varphi_k$ is convex, $|p_k|<4$ (see the proof of
Lemma~\ref{lem:convex} below). For fixed $x$ the right-hand side of
\eqref{eq:normx} is a convex function of $(p_1,\ldots,p_n)$, so its supremum
over the polydisk $\{|p_k|\leq 4\}$ is attained on the torus $\{|p_k|=4\}$,
which is its set of extreme points. Writing $p_k=4e^{i\theta_k}$, the phases
can be absorbed into the $x_k$ without changing $|x_k|$, and we obtain
$\|S_f(z)(\vec v,\vec v)\|^2\le8\,Q_n(x)\le\gamma_n^2$, that is,
$\|S_f\|\le\gamma_n$.

To see that $\gamma_n$ is sharp, let $\varphi_k(\zeta)=\zeta/(1-\zeta)$ for
all $k$, which maps $\mathbb{D}$ onto a half-plane, and take $z=(r,\ldots,r)$
with $0<r<1$. Then $p_k=2(1+r)$ for all $k$, and \eqref{eq:normx} becomes
$\|S_f(z)(\vec v,\vec v)\|^2=2(1+r)^2Q_n(x)$. Since $x$ ranges over the whole
unit sphere of $\mathbb{C}^n$, we get $\|S_f(z)\|=\frac{1+r}{2}\,\gamma_n$,
which tends to $\gamma_n$ as $r\to1$.

For the bounds \eqref{eq:gammabounds}, evaluating $Q_n$ at
$x=(1,-1,0,\ldots,0)/\sqrt2$ and at $x=e_1$ gives $Q_n=\frac12$ and
$Q_n=\bigl(\frac{n-1}{n+1}\bigr)^2$, which yields the lower bound. For the
upper bound, write $x_k-\frac{2}{n+1}\sum_jx_j=\sum_ib_{ki}x_i$ with
$b_{kk}=\frac{n-1}{n+1}$ and $b_{ki}=-\frac{2}{n+1}$ for $i\ne k$. By the
Cauchy--Schwarz inequality,
\[
\Bigl|\sum_ib_{ki}x_i\Bigr|^2\le\sum_i|b_{ki}|^2
=\frac{(n-1)^2+4(n-1)}{(n+1)^2}=\frac{(n-1)(n+3)}{(n+1)^2},
\]
so that $Q_n(x)\le(n-1)(n+3)/(n+1)^2$ on the unit sphere, and the upper bound
follows; it is smaller than $2\sqrt2$ because $(n-1)(n+3)<(n+1)^2$. Finally,
when $n=2$, write $a=|x_1|$, $b=|x_2|$ and $t=ab\in[0,\tfrac12]$. Since
$x_1-\tfrac23(x_1+x_2)=\tfrac13(x_1-2x_2)$,
\[
Q_2(x)\leq\frac19\bigl[a^2(a+2b)^2+b^2(b+2a)^2\bigr]
=\frac19\bigl(1+4t+6t^2\bigr)\leq\frac12 ,
\]
so that $\gamma_2\le2$, and the lower bound gives $\gamma_2=2$.
\end{proof}

\begin{remark}
The constant $\gamma_n$ is defined by a finite-dimensional extremal problem,
but we do not know its exact value for $n\ge3$. A numerical maximization of
$Q_n$ gives
\[
\gamma_3\approx2.0217,\qquad \gamma_4\approx2.1772,\qquad
\gamma_5\approx2.2916,\qquad \gamma_6\approx2.3739,
\]
so that neither bound in \eqref{eq:gammabounds} is attained for these values
of $n$. Since $|p_k|<4$ and the right-hand side of \eqref{eq:normx} is
homogeneous of degree two in $(p_1,\ldots,p_n)$, we have
$\|S_f(z)\|<\gamma_n$ for every $z\in\mathbb{P}^n$, so the bound is never
attained at a point.
\end{remark}

When $n=1$ the symmetry of the indices in the definition of $S^k_{ij}f$
forces the tensor $S_f$ to vanish for every locally univalent function. Thus
Theorem \ref{thm:polydisk} has no one-variable counterpart, and the value
$\gamma_2=2$ coincides with the constant of Chuaqui, Duren and Osgood only
numerically; the proof uses the pre-Schwarzian estimate
$(1-|z|^2)|\varphi''/\varphi'|<4$ and not their theorem. The classical
Schwarzian derivative is carried by $S^0_f$, which we study in
Section~\ref{sec:S0}.

\subsection{The Roper--Suffridge Operator}

Convex mappings of the Euclidean unit ball $\mathbb{B}^n$ are characterized
by conditions on the full Jacobian matrix (see \cite{graham_kohr2003}) that do
not lend themselves to explicit computations of $\|S_f\|$. The
Roper--Suffridge operator produces an explicit family of convex mappings of
$\mathbb{B}^n$ from convex functions of the disk, and within this class the
norm of the Schwarzian operator can be computed exactly.

Let $\varphi$ be a locally univalent function in
$\mathbb{D}$ of the form $\varphi(z)=z+a_2z^2+\cdots$. The Roper--Suffridge
extension operator is defined by
\begin{equation}\label{roper-suff}
\Phi(\varphi)(z_1,\ldots,z_n)
=\left(\varphi(z_1),\,\sqrt{\varphi'(z_1)}\,z_2,
\ldots,\sqrt{\varphi'(z_1)}\,z_n\right),
\end{equation}
where the branch of the square root is chosen so that $\sqrt{\varphi'(0)}=1$.
If $\varphi$ is univalent in $\mathbb{D}$, then $\Phi(\varphi)$ is
univalent in $\mathbb{B}^n$. This operator was introduced in \cite{RS} in order
to construct convex mappings in the Euclidean ball in $\mathbb{C}^n$ from
convex functions in the unit disk; that is, if $\varphi$ is a convex function
in $\mathbb{D}$, then $\Phi(\varphi)$ is a convex mapping in $\mathbb{B}^n$. A comprehensive treatment of the Roper--Suffridge operator, including its role in the construction of starlike and convex mappings in higher dimensions, can be found in \cite[Ch.~11]{graham_kohr2003}.
The Jacobian matrix of $f=\Phi(\varphi)$ is
\[
Df =\begin{pmatrix}
  \varphi' & 0 & \cdots & 0 \\
  \frac{1}{2}z_2\frac{\varphi''}{\sqrt{\varphi'}} & \sqrt{\varphi'} & \cdots & 0 \\
  \vdots & \vdots & \ddots & \vdots \\
  \frac{1}{2}z_n\frac{\varphi''}{\sqrt{\varphi'}} & 0 & \cdots & \sqrt{\varphi'}
\end{pmatrix},
\]
and the Jacobian determinant is $J_f(z)=(\varphi'(z_1))^{(n+1)/2}$. A direct
computation shows that the matrices of the Schwarzian derivative $S^kf(z)$ are
given by
\[
S^kf(z)=\begin{pmatrix}
  \frac{1}{2}z_k S\varphi(z_1) & 0 & \cdots & 0 \\
  0 & 0 & \cdots & 0 \\
  \vdots & \vdots & \ddots & \vdots \\
  0 & 0 & \cdots & 0
\end{pmatrix},
\qquad k=2,\ldots,n,
\]
and $S^1f(z)$ is the zero matrix. Moreover, for any
$\vec{v}=(v_1,\ldots,v_n)$,
\[
S_f(z)(\vec{v},\vec{v}) =
\begin{pmatrix}
  0 \\
  \frac{1}{2}z_2 S\varphi(z_1)v_1^2 \\
  \vdots \\
  \frac{1}{2}z_n S\varphi(z_1)v_1^2
\end{pmatrix}.
\]

\begin{theorem}\label{thm:RS}
Let $\varphi$ be a locally univalent function in $\mathbb{D}$ of the form
$\varphi(z)=z+a_2z^2+\cdots$, and let $f=\Phi(\varphi)$ be its
Roper--Suffridge extension to $\mathbb{B}^n$, $n\geq 2$. Then, for
$z=(z_1,z')\in\mathbb{B}^n$ with $z'=(z_2,\ldots,z_n)$,
\begin{equation}\label{eq:normRS}
\|S_f(z)\|=\frac{(1-|z_1|^2)^{3/2}\,|z'|\,|S\varphi(z_1)|}{2\sqrt{n+1}},
\end{equation}
and therefore
\[
\|S_f\|=\frac{\|S\varphi\|}{2\sqrt{n+1}},\qquad
\|S\varphi\|=\sup_{z\in\mathbb{D}}(1-|z|^2)^2|S\varphi(z)|.
\]
In particular, if $\varphi$ is convex, then
\[
\|S_f\|\leq\frac{1}{\sqrt{n+1}},
\]
and $\|S_f\|=1/\sqrt{n+1}$ when $\varphi$ maps $\mathbb{D}$ onto a parallel
strip.
\end{theorem}

\begin{proof}
By the formula for $S_f(z)(\vec v,\vec v)$ above,
\[
\|S_f(z)(\vec v,\vec v)\|=\frac12\,|S\varphi(z_1)|\,|v_1|^2\,\|(0,z')\|,
\]
where the norm is the Bergman norm at $z$. By \eqref{bergman metric},
\[
\|(0,z')\|^2=\frac{n+1}{(1-|z|^2)^2}\Bigl[(1-|z|^2)|z'|^2+|z'|^4\Bigr]
=\frac{(n+1)(1-|z_1|^2)\,|z'|^2}{(1-|z|^2)^2}.
\]
It remains to maximize $|v_1|^2$ over the vectors of Bergman norm one. The
matrix of the metric is
$G=\frac{n+1}{(1-|z|^2)^2}\bigl[(1-|z|^2)I+\bar z\,z^t\bigr]$, its inverse is
$G^{-1}=\frac{1-|z|^2}{n+1}\bigl[I-\bar z\,z^t\bigr]$, and for a positive
definite Hermitian form the maximum of $|v_1|^2$ on its unit sphere is the
first diagonal entry of the inverse matrix. Hence
\[
\max_{\|\vec v\|=1}|v_1|^2=\frac{(1-|z|^2)(1-|z_1|^2)}{n+1}.
\]
Multiplying these expressions we obtain \eqref{eq:normRS}. Since
$|z'|^2<1-|z_1|^2$, and $|z'|^2$ can be taken arbitrarily close to
$1-|z_1|^2$,
\[
\sup_{z'}\|S_f(z_1,z')\|=\frac{(1-|z_1|^2)^2\,|S\varphi(z_1)|}{2\sqrt{n+1}},
\]
and the formula for $\|S_f\|$ follows. If $\varphi$ is convex, then
$\|S\varphi\|\leq2$ by the theorem of Chuaqui, Duren and Osgood
\cite{chuaqui2011}, with equality for the strip mappings.
\end{proof}

\begin{remark}
The maximum of $|v_1|$ is not attained at the vector $(v_1,0,\ldots,0)$,
except when $z_1=0$ or $z'=0$. The optimal vector has a component in the
direction of $z'$, a reflection of the anisotropy of the Bergman metric of
the ball. For instance, if $\varphi(\zeta)=\frac12\log\frac{1+\zeta}{1-\zeta}$,
so that $S\varphi(\zeta)=2(1-\zeta^2)^{-2}$, then \eqref{eq:normRS} gives
$\|S_f(0,t,0,\ldots,0)\|=t/\sqrt{n+1}$ for $0<t<1$.
\end{remark}

\section{The quadratic form $S^0_f$}\label{sec:S0}

The coefficients $S^0_{ij}f$ define the quadratic form
$S^0_f(z)(\vec v)=\vec v^{\,t}S^0_f(z)\vec v$, and we set
\[
\|S^0_f(z)\|=\sup_{\|\vec v\|=1}|S^0_f(z)(\vec v)|,\qquad
\|S^0_f\|=\sup_{z\in\Omega}\|S^0_f(z)\|,
\]
where $\|\vec v\|$ is the Bergman norm at $z$, as in Section~\ref{sec:prelim}.
Since $u_0=J_f^{-1/(n+1)}$ is multiplied by a constant when $f$ is composed
with a nonsingular linear map, $S^0_f$ is invariant under such compositions,
and in the polydisk we may again assume that $f=(\varphi_1(z_1),\ldots,
\varphi_n(z_n))$. When $n=1$ and $\varphi$ is a function of one variable,
the Bergman metric of the disk is $2|dz|^2/(1-|z|^2)^2$, and since
$S^0_{11}\varphi=-\frac12S\varphi$ we have $\|S^0_\varphi\|=\frac14\|S\varphi\|$.
In this normalization the theorem of Chuaqui, Duren and Osgood reads
$\|S^0_\varphi\|\le\frac12$ for convex $\varphi$.

We begin with an estimate for convex functions of one variable. For a
locally univalent function $\varphi$ in $\mathbb{D}$ and $z\in\mathbb{D}$ we
write
\[
p_\varphi(z)=(1-|z|^2)\frac{\varphi''(z)}{\varphi'(z)},\qquad
s_\varphi(z)=(1-|z|^2)^2S\varphi(z).
\]

\begin{lemma}\label{lem:convex}
If $\varphi$ is convex in $\mathbb{D}$, then
$|p_\varphi(z)|^2+2|s_\varphi(z)|<16$ for all $z\in\mathbb{D}$.
\end{lemma}

\begin{proof}
Fix $z\in\mathbb{D}$ and let $T(\zeta)=(\zeta+z)/(1+\bar z\zeta)$. The Koebe
transform
\[
F(\zeta)=\frac{\varphi(T(\zeta))-\varphi(z)}{(1-|z|^2)\varphi'(z)}
=\zeta+A_2\zeta^2+A_3\zeta^3+\cdots
\]
is convex, with $A_2=\frac12p_\varphi(z)-\bar z$, and since $ST=0$ and
$T'(0)=1-|z|^2$ we have $s_\varphi(z)=SF(0)=6(A_3-A_2^2)$. By Trimble's
inequality \cite{Tri75}, $|A_3-A_2^2|\le(1-|A_2|^2)/3$, and $|A_2|\le1$ since
$F$ is convex. Hence $p_\varphi(z)=2(A_2+\bar z)$ and
$|s_\varphi(z)|\le2(1-|A_2|^2)$, so that
\begin{align*}
|p_\varphi(z)|^2+2|s_\varphi(z)|&\le4(|A_2|+|z|)^2+4(1-|A_2|^2)\\
&<4(1+|A_2|)^2+4(1-|A_2|^2)=8(1+|A_2|)\le16,
\end{align*}
which proves the lemma.
\end{proof}

The separate estimates $|p_\varphi|\le4$ and $|s_\varphi|\le2$ would only give
$|p_\varphi|^2+2|s_\varphi|\le20$. The point of Lemma~\ref{lem:convex} is that
the two bounds cannot be approached simultaneously: $|p_\varphi(z)|$ is close
to $4$ only if both $|A_2|$ and $|z|$ are close to $1$, and then
$|s_\varphi(z)|\le2(1-|A_2|^2)$ is small, while $|s_\varphi(z)|=2$ forces
$A_2=0$ and hence $|p_\varphi(z)|=2|z|<2$. This joint
estimate is what makes the constant in Theorem~\ref{thm:S0poly} sharp.

For $f=(\varphi_1(z_1),\ldots,\varphi_n(z_n))$ we write
$q_i=\varphi_i''/\varphi_i'$ and $S_i=S\varphi_i$, evaluated at $z_i$.
Since $u_0=\prod_i(\varphi_i')^{-1/(n+1)}$ and, by the definition of
$S^k_{ij}f$, $S^k_{ij}f=\delta_{ij}\delta_{ik}q_i-\frac1{n+1}
(\delta^k_iq_j+\delta^k_jq_i)$, a direct computation gives
\begin{equation}\label{eq:S0poly}
S^0_{ii}f=-\frac{S_i}{n+1}+\frac{(n-1)\,q_i^2}{2(n+1)^2},\qquad
S^0_{ij}f=-\frac{q_iq_j}{(n+1)^2}\quad(i\ne j).
\end{equation}
When $n=1$ the second term vanishes and we recover
$S^0_{11}\varphi_1=-\frac12S\varphi_1$.
For $n\ge2$ the pre-Schwarzian terms remain, and they determine the sharp
bound.

\begin{theorem}\label{thm:S0poly}
Let $n\ge2$ and let $f:\mathbb{P}^n\to\mathbb{C}^n$ be a convex biholomorphic
mapping. Then
\[
\|S^0_f\|\le\frac{4}{n+1},
\]
and the constant is sharp, although it is not attained.
\end{theorem}

\begin{proof}
As in the proof of Theorem~\ref{thm:polydisk}, write
$v_i=(1-|z_i|^2)x_i/\sqrt2$, so that $\|\vec v\|=1$ if and only if
$\sum_i|x_i|^2=1$, and let $p_i=p_{\varphi_i}(z_i)$, $s_i=s_{\varphi_i}(z_i)$.
By \eqref{eq:S0poly},
\[
S^0_f(z)(\vec v)=\frac{1}{4(n+1)^2}\Bigl[-2(n+1)\sum_{i=1}^ns_ix_i^2
+(n+1)\sum_{i=1}^np_i^2x_i^2-2\Bigl(\sum_{i=1}^np_ix_i\Bigr)^2\Bigr].
\]
The last two terms equal $y^tBy$, where $y_i=p_ix_i$ and
$B=(n+1)I-2\mathbf{1}\mathbf{1}^t$. The real symmetric matrix $B$ has
eigenvalues $n+1$ and $1-n$, so $\|B\|=n+1$ and
$|y^tBy|\le(n+1)|y|^2=(n+1)\sum_i|p_i|^2|x_i|^2$. Therefore, by
Lemma~\ref{lem:convex},
\[
|S^0_f(z)(\vec v)|\le\frac{1}{4(n+1)}\sum_{i=1}^n|x_i|^2\bigl(|p_i|^2+2|s_i|\bigr)
<\frac{4}{n+1}.
\]
For sharpness, let $\varphi_i(\zeta)=\zeta/(1-\zeta)$ for all $i$, so that
$s_i=0$, and take $z=(r,\ldots,r)$ and $x=(1,-1,0,\ldots,0)/\sqrt2$. Then
$p_i=2(1+r)$ and $\sum_ip_ix_i=0$, so that
$S^0_f(z)(\vec v)=(1+r)^2/(n+1)\to4/(n+1)$ as $r\to1$.
\end{proof}

The extremal configuration differs from the one-variable situation in two
respects. The half-plane mapping, which has vanishing Schwarzian derivative,
replaces the parallel strip, and the bound is not attained. The norm is
produced by the pre-Schwarzian terms in \eqref{eq:S0poly}, both diagonal and
off-diagonal, which vanish when $n=1$; at the extremal vector
$x=(1,-1,0,\ldots,0)/\sqrt2$ the diagonal terms contribute $(n-1)p^2$ and the
off-diagonal terms $2p^2$ to the quantity $(n+1)p^2$ in the proof, where $p$
is the common value of the $p_i$.

For extensions by the Roper--Suffridge operator the situation is different:
the classical Schwarzian derivative appears directly.

\begin{theorem}\label{thm:S0RS}
Let $\varphi$ be a locally univalent function in $\mathbb{D}$ of the form
$\varphi(z)=z+a_2z^2+\cdots$, and let $f=\Phi(\varphi)$, $n\ge2$. Then
$S^0_{11}f(z)=-\frac12S\varphi(z_1)$ and $S^0_{ij}f=0$ for all other
$i,j$. Consequently
\[
\|S^0_f(z)\|=\frac{(1-|z|^2)(1-|z_1|^2)\,|S\varphi(z_1)|}{2(n+1)},\qquad
\|S^0_f\|=\frac{\|S\varphi\|}{2(n+1)}.
\]
In particular, if $\varphi$ is convex, then $\|S^0_f\|\le1/(n+1)$, with
equality when $\varphi$ maps $\mathbb{D}$ onto a parallel strip.
\end{theorem}

\begin{proof}
Since $J_f=(\varphi'(z_1))^{(n+1)/2}$, we have $u_0=(\varphi'(z_1))^{-1/2}$,
which depends only on $z_1$. Writing $q=\varphi''/\varphi'$, we get
$\partial u_0/\partial z_1=-\frac12qu_0$ and
$\partial^2u_0/\partial z_1^2=(\frac14q^2-\frac12q')u_0$, and all other
derivatives of $u_0$ vanish. Since $S^1f=0$, the definition of $S^0_{ij}f$
gives $S^0_{11}f=\frac14q^2-\frac12q'=-\frac12S\varphi(z_1)$, and
$S^0_{ij}f=0$ otherwise. Hence $S^0_f(z)(\vec v)=-\frac12S\varphi(z_1)v_1^2$,
and by the proof of Theorem~\ref{thm:RS} the maximum of $|v_1|^2$ over the
unit sphere of the Bergman metric at $z$ is $(1-|z|^2)(1-|z_1|^2)/(n+1)$. This
gives the formula for $\|S^0_f(z)\|$. For fixed $z_1$ it is largest when
$z'=0$, so that
$\|S^0_f\|=\sup_{z_1}(1-|z_1|^2)^2|S\varphi(z_1)|/(2(n+1))$. The last
statement follows from the theorem of Chuaqui, Duren and Osgood
\cite{chuaqui2011}. For the strip mapping
$\varphi(\zeta)=\frac12\log\frac{1+\zeta}{1-\zeta}$ we have $S\varphi(0)=2$,
so the bound is attained at the origin.
\end{proof}

For $n=1$ the formula of Theorem~\ref{thm:S0RS} reduces to
$\|S^0_\varphi\|=\frac14\|S\varphi\|$, which is the one-variable statement in the
normalization described at the beginning of this section. Thus, in the class of
Roper--Suffridge extensions, the bound for $\|S^0_f\|$ is a genuine extension
of the theorem of Chuaqui, Duren and Osgood, while Theorem~\ref{thm:S0poly}
shows that this is not the case for general convex mappings of the polydisk.

% --- Bibliografía ---
\bibliographystyle{amsplain}
\bibliography{schwarziana}

\end{document}